\documentclass[letterpaper,10pt,conference]{ieeeconf}
\IEEEoverridecommandlockouts 

\usepackage{amsmath}
\usepackage{amsfonts}
\usepackage{amssymb}
\usepackage{verbatim}
\usepackage{enumerate}

\newcommand{\rset}{\mathbb{R}}
\newcommand{\bo}{\mathbf}

\newtheorem{definition}{Definiton}

\newtheorem{assumption}{Assumption}
\newtheorem{theorem}{Theorem}
\newtheorem{proposition}{Proposition}
\newtheorem{remark}{Remark}
\title{\LARGE \bf
Distributed model predictive control of leader-follower systems using
an interior point method with efficient computations}

\author{ Ion Necoara, Dragos N. Clipici and Sorin Olaru
\thanks{I. Necoara and  D.N. Clipici are with  University  Politehnica Bucharest,
Automatic Control and Systems Engineering Department, 060042
Bucharest, Romania. {\tt\small \{ion.necoara,
d.clipici\}@acse.pub.ro}. S. Olaru is with SUPELEC, Automatic
Control Department, Gif sur Yvette, France
 {\tt\small sorin.olaru@supelec.fr}.}}

\begin{document}
\maketitle

\begin{abstract}
Standard model predictive control strategies imply the online
computation of control inputs at each  sampling instance, which
traditionally limits this type of control scheme to systems with
slow dynamics. This paper focuses on distributed  model predictive
control for large-scale systems comprised of interacting linear subsystems,
where the online computations required for the control input can be
distributed amongst them. A model predictive controller based on a
distributed interior point method is derived, for which every
subsystem in the network can compute  stabilizing control inputs
using  distributed computations. We introduce local terminal sets
and cost functions, which together satisfy distributed invariance
conditions for the whole system, that guarantees   stability of the
closed-loop interconnected system. We show that the synthesis of
both terminal sets and terminal cost functions can be done in a
distributed framework.
\end{abstract}

%%%%%%%%%%%%%%%%%%%%%%%%%%%%%%%%%%%%%%%%%%%%%%%%%%%%%

\section{Introduction}
\label{introducere}

Model predictive control (MPC) is a well established method of
process control that has proven to be useful in numerous industrial
applications in the past decades. One of the advantages of MPC is
that it can be applied to large scale systems, with a considerable
number of states and inputs for which hard constraints are often
required \cite{c17,c16}.

MPC requires that the control input at each time step be calculated
by the online solution of an  optimization problem. As a result, one
of the drawbacks of MPC as a control algorithm is the delay
introduced by the computation time that is imposed for the
evaluation of functions, their first or second order derivatives,
and for matrix operations, computations that are usually required
for most optimization algorithms. This computational burden is also
worsened when MPC is implemented for a large-scale  plant of
interconnected subsystems, case where the dimension of the MPC
 problem is multiplied by the number of subsystems. For certain
industries for which the manufacturing process is slow in nature,
this computational time is not an issue.

However, multi-system applications have arisen where computing the
input rapidly is essential for efficiency and stability. Control
problems for networks of interconnected multi-agent systems such as
traffic control \cite{c15}, building anti-earthquake systems
\cite{c21} , satellite formation flight \cite{c23}, and wind turbine
farms \cite{c19}, have received plenty of interest in recent years.
Due to the large number of inputs and outputs of this class of
systems, distributed control is often required. Efficient
distributed optimization methods for solving such control problems
can be found in \cite{c1,c8,NecJPC,NecLNCS,c25}. From a practical
viewpoint, such methods can be sped up by implementing stronger,
more powerful computational hardware.
 %However, this may prove to be
%both costly in financial terms and physically unfeasible due to
% the reduced dimensions of some of these systems.
Recent results in \cite{c4,c9,JerCon:11,c24,NecJPC,NecLNCS,c3}
however, have shown that by exploiting the special underlying
structure of some MPC problems, the number of flops required for an
algorithm can be reduced substantially,  thus making MPC a more
attractive solution for control problems where speed is essential.
The authors in \cite{c9} propose an interior point method approach
for solving the MPC problem in which they use a discrete-time
Riccati recursion to solve the linear equations efficiently at each
iterate. In \cite{c4} the authors propose a more efficient approach
to linear algebra computations w.r.t the derivation given in
\cite{c3}. Computational burden can be also overcome by distributing
the necessary operations amongst different agents. To this purpose,
the authors in \cite{NecJPC} examine a distributed approach to
optimal control problems and appropriate optimization methods.

In this paper, we focus on extending these recent results on the
computational time required for the control action for MPC problems
with a special underlying structure arising in large-scale
leader-follower systems, where the computational burden is
distributed amongst the comprising subsystems, thus providing a
certain independence that is usually required for these subsystems.
In the first part of the paper, a stability analysis for
leader-follower systems is presented,  based on a linear feedback
law, that allows us to construct local terminal sets and cost
functions in a completely distributed way. Compared with the
existing approaches based on an end point constraint, we reduce the
conservatism by combining the underlying structure of the system
with distributed optimization. This leads to a larger region of
attraction for the controller. Then,  we formulate a distributed MPC
problem for this type of systems, using a terminal cost-terminal set
approach and an efficient implementation of an interior-point
algorithm using Mehrothra's predictor-corrector scheme for solving
the corresponding optimization problem is presented. In particular,
we show how the underlying Newton system can be solved in a
distributed manner.

The paper is organized as follows. In Section \ref{problema} we
present the formulation of the MPC problem corresponding to systems
of the leader-follower type and then we investigate the stability
issue for the current system in a distributed manner via a linear
feedback law using  a structured  Lyapunov function approach. In
Section \ref{sec2} we focus on decomposing the terminal state
constraints required for stability as a Cartesian product using
distributed set computations, after which we formulate the general
centralized MPC problem. We then show how to restructure the
original MPC problem in Section
 \ref{restructurare} as to provide computational benefits using
a distributed version of an interior-point algorithm presented in Section \ref{ipm}.

%%%%%%%%%%%%%%%%%%%%%%%%%%%%%%%%%%%%%%%%%%%%%%%%%%%%%%%%%%%%%%%%

\section{Distributed MPC using the terminal-cost, terminal set approach}
\label{problema}

Large scale systems have attracted much interest from the control
systems community in recent decades.  In this paper, we focus on
large scale systems of leader-follower type. The MPC problem
associated with leader-follower systems can be found in a number of
current applications such as  platoons of vehicles \cite{c7}, which
is of great interest in the development of automated highway systems
 \cite{c15}, or in the renewable energy industry such as the
problem of controlling a wind turbine farm \cite{c19}. \\ \indent
Platoon or leader-follower systems imply that each subsystem, from
the second one onwards, is influenced by the previous. We consider
linear time invariant systems, for which the dynamics of the first
subsystem are:
\begin{align}
x_{t+1}^1=A^1 x_t^1 + B^1 u_t^1 \label{prob_first}.
\end{align}
The dynamics for the remaining $M-1$ subsystems are described by the
following linear equations:
\begin{align}
&x_{t+1}^i=A^i x_t^i + B^i u_t^i + A^{i, i-1} x_t^{i-1}+B^{i,i-1} u_t^{i-1},   \label{prob_din}
\end{align}
where $x_t^i \in \rset^{n_i}$ and $u_t^i \in \rset^{m_i}$ are the
state and input vectors of subsystem $i$ at  time $t$, $A^i \in
\rset^{n_i \times n_i}$ and $B^i \in \rset^{n_i \times m_i}$ are the
state and input dynamic matrices for subsystem $i$, while $A^{i,i-1}
\in \rset^{n_i \times n_{i-1}}$  and $B^{i,i-1} \in \rset^{n_i
\times m_{i-1}}$ are the matrices for the coupling dynamics which
define the influence of subsystem $i-1$ upon subsystem $i$. For
these systems, we consider mixed state and input constraints of the
following polyhedral form:
\begin{equation}\label{cstr}
G_x^i x_t^i + G_u^i u_t^i \leq b^i, %\nonumber
\end{equation}
where  $ G_x^i  \in \rset^{q_i \times n_i}$, $ G_u^i  \in \rset^{q_i
\times m_i}$, the  matrix $\begin{bmatrix} G_x^i & G_u^i
\end{bmatrix} \in \rset^{q_i \times n_i+m_i}$ has full row rank and $b^i>0$. We
employ stage cost functions for states and inputs of the quadratic
form\footnote[1]{In this paper, we use the following notation:
$\left \| x \right \|^2_P=x^T P x$. }:
\begin{equation}
\ell^i(x_t^i,u_t^i)=\frac{1}{2} \left (\left \|x_t^i \right \|^2_{Q^i} + \left \| u_t^i \right \|^2_{R^i} \right ), \nonumber \\
\end{equation}
where $Q^i \in \rset^{n_i \times n_i}$ and $R^i \in \rset^{m_i \times m_i}$ are positive definite.
For the stability analysis, we also express the dynamics for the entire system as follows:
\begin{equation}
\bo{x}_{t+1}=\bo{A} \bo{x}_{t}+\bo{B} \bo{u}_t \label{prob_big},
\end{equation}
where $\bo{x}_t \in \rset^n$ and $\bo{u}_t \in \rset^m$ comprise the
states and inputs of all  the subsystems at time $t$ and the
matrices $\bo{A}$ and $\bo{B}$ are block banded matrices comprised
of $A^i$, $A^{i,i-1}$ and $B^i$, $B^{i,i-1}$ respectively. In a
similar fashion we define the  block diagonal matrices $\bo{Q_d}$ and
$\bo{R_d}$ comprised of $Q^i$ and $R^i$, respectively. In order to
ensure stability for the MPC scheme that we define below, we use a
terminal set-terminal cost approach \cite{c17,c16}. We define the
following final stage cost of the form:
\begin{equation}
\ell_{\text{f}}(\bo{x})=\left \| \bo{x}  \right \|^2 _{\bo{P_d}},
\nonumber
\end{equation}
where matrix $\bo{P_d} \in \rset^{n \times n}$  is positive definite.
In order to find $\bo{P_d}$ and also a terminal set $X_{\text{f}}$ we
search for a linear feedback  law $\bo{u}_t=\bo{K_d}\bo{x}_t$, such that
the system
\begin{equation}\label{system_stabil}
\bo{x}_{t+1}=(\bo{A+BK_d}) \bo{x}_t
\end{equation}
satisfies the following three properties \cite{c17}:
\begin{enumerate}
\item[A.1]{$\left \{ \left (\bo{x},\bo{K_dx}\right )| \bo{x} \in X_{\text{f}} \right \} \subseteq  \left \{ (\bo{x},\bo{u}) | G_x^ix^i+G_u^i u^i \leq b^i\; \right \} $}
\item[A.2]{$(\bo{A+BK_d}) \bo{x} \in X_{\text{f}}$, $ \forall \bo{x} \in X_{\text{f}}$}
\item[A.3]{$\ell_\text{f}$ satisfies the following property:}
\end{enumerate}
\begin{align}
\ell_\text{f}((\bo{A+BK_d})\bo{x})-&\ell_\text{f}(\bo{x})+\bo{x^TK_d^TR_dK_dx}+\bo{x^TQ_dx}\leq
0, \nonumber \\ &\forall \bo{x} \in X_{\text{f}}. \label{lyap1}
\end{align}
The centralized MPC scheme for the leader-follower system  described
by dynamics \eqref{prob_first}-\eqref{prob_din} based on a terminal
set-terminal cost approach, given an initial state $\bo{x}$ and
prediction horizon $N$, is formulated as follows:
\begin{align}
V_N(\bo{x})=&\min_{\bo{x},\bo{u}} \sum_{i=1}^M  \sum _{t=0}^{N-1} \ell^i(x_t^i,u_t^i) + \ell_\text{f}(\bo{x}_N)   \nonumber \\
\text{s.t:} \; &\text{dynamics \eqref{prob_first} and \eqref{prob_din}} \label{scheme_first} \\
& G_x^i x_t^i  + G_u^i u_t^i \leq b^i, x_0^i=x^i, \forall i=1,\dots,M, \nonumber \\
&   \bo{x}_N \in X_{\text{f}}. \nonumber
\end{align}
It is a well-known result \cite{c17} that the above MPC scheme,
under assumptions A.1-A.3,  stabilizes the system \eqref{prob_big},
with the optimal value of problem \eqref{scheme_first},
$V_N(\bo{x})$, as a Lyapunov function. Keeping in line with the
distributed nature of our system, the control law $\bo{K_d}$, the
final stage cost $\ell_\text{f}$ and the terminal constraint set
$X_{\text{f}}$ need to be computed locally. In the following
sections we develop a distributed synthesis  procedure under such
structural constraints.

%%%%%%%%%%%%%%%%%%%%%%%%%%%%%%%%%%%%%%%%%%%%%

\subsection{Terminal Cost} \label{sec1}
For a locally computed $\bo{K_d}$, we employ distributed control laws
$u^i=K^i x^i$ for each subsystem,  with $K^i \in \rset^{m_i \times
n_i}$ and the resulting control law for the entire system will then
be $\bo{u}= \bo{K_d} \bo{x}$, where the matrix
$\bo{K_d}=\text{diag}(K^i)$ is block-diagonal. For the terminal stage
cost, we define $\ell_\text{f}(\bo{x})=\displaystyle \sum_{i=1}^M
\ell_\text{f}^i (x^i)$, where terminal stage costs for each
subsystem are of the following quadratic form:
\begin{equation}
\ell_\text{f}^i(x^i)=\frac{1}{2} \left \| x^i \right \|^2_{P^i}, \; \forall i=2,\dots,M, \nonumber
\end{equation}
where the matrix $P^i \in \rset^{n_i \times n_i}$ is positive definite, such that $\bo{P_d}=\text{\text{diag}}(P^i)$.

Due to the block-diagonal structure of matrices $\bo{P_d}$, $\bo{Q_d}$
and $\bo{R_d}$, we  can rewrite \eqref{lyap1} equivalently as the
following inequality:
\begin{eqnarray}
&&V_f^1(x^1) +\sum_{i=2}^{M} V_f^i(x^i,x^{i-1}) \leq 0, \; \forall
\bo{x} \in X_{\text{f}} ,\nonumber
\end{eqnarray}
where the left hand side is a sum of local functions $V_f^i$ that have the following form:
\begin{align}
&V_f^1(x^1) \!=\! (x^1)^T \left ( (\tilde{A}^1)^T P^1 \tilde{A}^1 \!-\! P^1 \!+\! Q^1 \!+\!
 (K^1)^T R^1 K^1 \right ) x^1 \nonumber \\
&V_f^i(x^i,x^{i-1})=\begin{bmatrix}  (x^{i})^T & (x^{i-1})^T \end{bmatrix}
 \bo{P}^i  \begin{bmatrix} x^i\\ x^{i-1}  \end{bmatrix}, \; \forall i \geq 2, \nonumber
\end{align}
where $\tilde{A}^i=A^i+B^i K^i$,
$\tilde{A}^{i,i-1}=A^{i,i-1}+B^{i,i-1}K^{i-1}$  and matrices $\bo{P}^i$ are of
the following form:
\begin{align}
&\begin{bmatrix} (\tilde{A}^i)^T P^i \tilde{A}^i \!-\! P^i \!+\! Q^i \!+\!
(K^i)^T R^i K^i & \!\!\! (\tilde{A}^{i})^T P^i \tilde{A}^{i,i-1}  \\
(\tilde{A}^{i-1,i})^T P^i \tilde{A}_{i} & \!\!\!
(\tilde{A}^{i,i-1})^T P^i \tilde A^{i,i-1} \end{bmatrix}. \nonumber
\end{align}
We can ensure inequality \eqref{lyap1} imposing the following
distributed structure (see also \cite{c8} for a similar approach):
\begin{align}
V_f^1(x^1) &\leq q^1(x^1) \IEEEyessubnumber \label{first1} \\
V_f^i(x^i,x^{i-1}) & \leq q^i (x^i,x^{i-1}), \;   \forall
i=2,\dots,M, \;  \bo{x} \in X_{\text{f}}  \IEEEyessubnumber
\label{first2}
\end{align}
such that:
\begin{equation}
q(\bo{x}) = q^1(x^1) + \sum_{i=2}^M q_i(x^i,x^{i-1}) \leq 0, \;
\forall \bo{x} \in X_{\text{f}}. \label{sum_neg}
\end{equation}
We consider that the functions $q^i$ do not necessarily take
negative values  and have the following quadratic form:
\begin{align}
q^1(x^1)&= (x^1)^T W^1 x^1 \nonumber\\
q^i(x^i,x^{i-1})&= \begin{bmatrix} (x^i)^T & (x^{i-1})^T
\end{bmatrix}  W^i   \begin{bmatrix} x^i \\ x^{i-1} \end{bmatrix},
\nonumber
\end{align}
where the matrices $W^i=\begin{bmatrix} (W^i)_{11} & (W^i)_{12} \\
(W^i)_{12}^T & W^i_{22}   \end{bmatrix}$ are symmetric. Clearly,
$q(\bo{x})$ is also quadratic function and thus can be written as
$q(\bo{x}) = \bo{x}^T W \bo{x}$, for an appropriate matrix $W$
defined below. We now define the following optimization problem:
\begin{align}\label{sdp_1}
 &\underset{P^i,K^i,W^i,\tau}{\min} \tau  \\
\text{ s.t:} \; & \text{MI}^i(P^i,K^i,W^i) \preccurlyeq 0 \text{, } \forall i=1,\dots, M\nonumber \\
&W \preccurlyeq \tau I, \nonumber
\end{align}
where $\text{MI}^i(\cdot)$ refer to the matrix inequalities
\eqref{first1} and \eqref{first2} and  the matrix $W$ has the
following block tridiagonal structure:
\begin{equation}\nonumber
\begin{bmatrix}
 W^1+W^2_{22} & W^2_{12} & 0 & \dots & 0
\\ (W^2_{12})^T & W^2_{11}+W^3_{22} & \dots & 0 & \vdots
\\ 0 & \ddots &  \ddots & \ddots & 0
\\\vdots & 0 &\dots & W^{M-1}_{11}+W^M_{22} & W^M_{12}
 \\ 0& \dots & 0 &  (W^M_{12})^T & W^M_{11}\end{bmatrix}.
\end{equation}
It is straightforward to see that if \eqref{sdp_1} has an optimal
value $\tau^* \leq 0$,  ensuring that $W\leq0$ and subsequently
\eqref{sum_neg} holds, then \eqref{lyap1} is satisfied. Note that we
do not require that matrices $W_i$ to be negative semi-definite. On
the contrary, positive or indefinite matrices allow local terminal
costs to increase as long as the global cost still decrease.  This
approach reduces conservatism in deriving the matrices $P_i$ and
$K_i$. However, problem \eqref{sdp_1} is in a form that cannot be
solved efficiently since it is not a convex problem. Subsequently,
we show that \eqref{sdp_1} can be expressed as a sparse SDP that can
be solved distributively. To this goal, we employ the following
linearization \cite{c16}:
\begin{eqnarray}
 \label{lin1}
P^i = (S^i)^{-1}, \quad K^i = Y^i G^{-1}.
\end{eqnarray}
We also define the following matrices in order to make constraints
of the optimization problem in the following theorem more compact
notationally:
\begin{align}\nonumber
\tilde{G}^1&= G+G^T- S^1+\tilde{W}^1 \nonumber \\
\tilde{G}^i&= \begin{bmatrix} G+G^T -S^i+\tilde{W}^i_{11} & \tilde{W}^i_{12} \\
(\tilde{W}^i_{12})^T & G+G^T -\mu^i I +\tilde{W}^i_{22} \nonumber \end{bmatrix} \\
\tilde{B}^1 & = \begin{bmatrix} A^1 G + B^1 Y^1 \\
(Q^1)^{\frac{1}{2}} G  \\
(R^1)^{\frac{1}{2}}Y^1
\end{bmatrix}, \nonumber \\
\tilde{B}^i & = \begin{bmatrix} A^iG+B^iY^i &
{A}^{i,i-1}G+B^{i,i-1}Y^{i,i-1} \\ (Q^i)^{\frac{1}{2}} G  & 0   \\
(R^i)^{\frac{1}{2}}Y^i & 0
\\  0 & G^T    \end{bmatrix} \nonumber \\
\tilde{S}^1&=\begin{bmatrix} S^1 & 0 & 0   \\ 0 & I & 0  \\ 0 & 0 & I   \end{bmatrix}, \;
 \tilde{S}^i=\begin{bmatrix} S^i & 0 & 0 & 0 \\ 0 & I & 0 & 0 \\ 0 & 0 & I & 0 \\ 0& 0& 0 & \mu^i I  \end{bmatrix}. \nonumber
\end{align}

Note that the linearizations \eqref{lin1} have been employed under
the  assumption that all the subsystems have the same dimension for
the states, i.e. $n_i = n_j$ for all $i, j$.

\begin{theorem}
If the following SDP
\begin{eqnarray}\label{sdp_2}
 &\underset{G, S^i, Y^i, Y^{i,i-1}, \tilde{W}, \mu^i, \tau} \min \; \tau  \\
\text{ s.t: } \; & \begin{bmatrix} \tilde{G}^i & (\tilde{B}^i)^T \\ \tilde{B}^i & \tilde{S}^i
\end{bmatrix} \succcurlyeq 0, \; \forall
i=1,\dots,M \label{lmi2} \\
& Y^{i,i-1}=Y^{i-1},  \; \forall
i=2,\dots,M  \nonumber \\
& \tilde{W} \preccurlyeq \tau I, \nonumber
\end{eqnarray}
where $\tilde{W}$ has the same structure as $W$, has a negative
optimal value $\tau^* < 0$,  then \eqref{lyap1} holds.
\end{theorem}
\begin{proof}
From \eqref{lmi2} we observe that $S^i \succ 0$ and $\mu^i
>0$, which in turn implies:
\begin{align}
(S^i-G)^T &(S^i)^{-1} (S^i-G) \succcurlyeq 0 \nonumber \\
(\mu^i I-G)^T &\frac{1}{\mu^i}I (\mu^i I-G) \succcurlyeq 0.
\nonumber
\end{align}
By adding $\tilde{W}^i$ to the previous inequalities, we get:
\begin{align}\label{lmi3}
\tilde{G}^i \preccurlyeq \begin{bmatrix} G^T & 0 \\ 0 & G^T \end{bmatrix} \begin{bmatrix} (S^i)^{-1} & 0 \\
 0 & \frac{1}{\mu^i} I \end{bmatrix}
\begin{bmatrix} G & 0 \\ 0 & G \end{bmatrix} + \tilde{W}^i.
\end{align}
For $i=2, \cdots, M$, using \eqref{lmi3} and the equality
constraints $Y^{i,i-1}=Y^{i-1}$ and by applying the Schur complement
to \eqref{lmi2} we obtain:
\begin{align}\nonumber
\tilde{W}^i   \succcurlyeq & \begin{bmatrix} (\tilde{A}^i)^T P^i \tilde{A}^i -P^i  & (\tilde{A}^i)^T P^i \tilde{A}^{i,i-1} \nonumber \\
(\tilde{A}^{i,i-1})^T P^i \tilde{A}^i & (\tilde{A}^{i,i-1})^T P^i
\tilde{A}^{i,i-1}
 \end{bmatrix} \\& + \begin{bmatrix}  Q^i +(K^i)^T R^i K^i +G^{-T}G^{-1} & 0 \\ 0 & 0 \end{bmatrix} \nonumber ,
\end{align}
which is equivalent to \eqref{first2} if we take:
\begin{align}\nonumber
{W}^i=\begin{bmatrix} G^{-T} & 0 \\ 0 & G^{-T} \end{bmatrix}
\begin{bmatrix} \tilde{W}^i_{11} & \tilde{W}^i_{12} \\ (\tilde{W}^i_{12})^T &
 \tilde{W}^i_{22}  \end{bmatrix}
\begin{bmatrix} G^{-1} & 0 \\ 0 & G^{-1} \end{bmatrix}.
\end{align}
To transform inequality \eqref{first1} into a linear matrix
inequality of  type \eqref{lmi2}, we use the same linearizations and
the proof follows similar steps as those previously presented. As a
result, the SDP \eqref{sdp_2} is equivalent to problem
\eqref{sdp_1}, and for a negative optimal value $\tau^*$,
\eqref{lyap1} is satisfied.
\end{proof}
Note that the SDP problem \eqref{sdp_2}-\eqref{lmi2} can be solved
offline either using a  sparse SDP solver or some distributed
optimization algorithm \cite{NecJPC}. Since we have imposed a
diagonal structure on the controller $\bo{K_d}=\text{diag}(K^i)$, it
follows that the system matrix $\bo{A+BK_d}$ has a block bidiagonal
structure. If the optimal solution $\tau^*$ of the SDP is negative,
then the matrix $\bo{A+BK_d}$ is Schur (all the eigenvalues are strict
inside the unit circle). It follows that all the matrices $A^i + B^i
K^i$ are Schur.

%%%%%%%%%%%%%%%%%%%%%%%%%%%%%%%%%%%%%%%%%%%%%%%%%%%%%%%%%%%%%%%%%%%%%%%%%

\subsection{Terminal Set}\label{sec2}
To complete the stability analysis for system \eqref{system_stabil},
which implies  properties A.1 - A.3, we need to complete the design
procedure  by the computation of a terminal set
$X_f\subset\rset^{n}$, defined locally (as a Cartesian product)
$X_f=\underset{i=1}{\overset{M}{{\large \Pi}}} X_f^i$ and equipped
with invariance properties.

First let us define the set of \emph{admissible states} associated
to  the constraints \eqref{cstr} and the specific linear controller
$K^i$:
\[ X^i =\{x^i: (G_x^i + G_u^i K^i) x^i \leq b^i  \}.   \]
leading via the Cartesian product to a set in $\rset^{n}$:
\[ X =  \underset{i=1}{\overset{M}{{\large \Pi}}} X^i.   \]
\begin{assumption}\label{asum} The origin is assumed to be an interior point of the set $X$.
\end{assumption}

We introduce the following formal definition of \emph{positive invariance} in view of its use in the practical construction of the terminal set $X_f$.
\begin{definition}
A set $ \Omega \subseteq X$ is called positive invariant for system
\eqref{system_stabil}  if $\bo{x}_t \in \Omega$ it holds that
$\bo{x}_{t+1} \in \Omega$ for all $t \geq 0$.
\end{definition}
As a standard approach in the MPC design \cite{c17}, the terminal set $X_{\text{f}}\subset X$
needs to be positive  invariant for the nominal linear time-invaraint dynamics \eqref{system_stabil}. This is a standard problem in set-theoretic control theory and there are a number of ways in which can be computed (see e.g
\cite{Gil,RakKerrigan,c18}). %\cite{Gil},

Due to the distributed nature of our system, such a general terminal
constraint set  cannot be used due to the introduction of coupling
constraints between the states of the subsystems. We need to explore the
possibility of finding a  terminal constraint set, which preserves
the structure of a Cartesian product:
 \begin{equation}
 X_{\text{f}} =  \underset{i=1}{\overset{M}{{\large \Pi}}} X_{\text{f}}^i \label{relax},
\end{equation}
This will further enable a distributed use of the terminal constraint sets $X_{\text{f}}^i$ for each
of the subsystems. Is worth mentioning that for general systems the construction of a
terminal set in the form given above  can be cumbersome in
distributed settings (see e.g. \cite{RakDec} for such a
construction). However, for a system $\bo{x}_{t+1} = \tilde{\bo{A}}
\bo{x}_t$, where $\tilde{\bo{A}}$ has a special block bidiagonal structure and the admissible set is expressed as $X =
\underset{i=1}{\overset{M}{{\large \Pi}}} X^i$, the computation of
such an invariant set $X_{\text{f}}= \Pi_{i=1}^M
X_{\text{f}}^i$ can be simplified  by exploiting these structural properties.% of
%the matrix  $\tilde{\bo{A}}$.

 Without loss of generality the matrix  $\tilde{\bo{A}}$ will be  considered to
  be of the following form:
\begin{equation*}
\tilde{\bo{A}}=\begin{bmatrix} \bar A^{11} & 0 & \dots & \dots & 0 \\
\bar A^{21} & \bar A^{22} &  0 & \dots & 0 \\ 0 & \bar A^{32} & \bar A^{33} & \dots & 0 \\
\vdots &  \vdots &  \ddots & \ddots & 0\\ 0 & 0 & 0 & \bar A^{M,M-1}
& \bar A^{MM} \end{bmatrix}
\end{equation*}
i.e block lower-bidiagonal. The developments in Subsection
\ref{sec1} point to the construction of a distributed linear
controller which allow us to assume the stability of the
unconstrained local closed-loop system
$\bo{x}_{t+1}=\tilde{\bo{A}}\bo{x}_t$ around the origin. By the
block  lower-bidiagonal structure it follows that the matrix
$\tilde{\bo{A}}$ is Schur (i.e. $\left | \lambda(\tilde{\bo{A}})
\right |<1 $) and consequently through the block lower-bidiagonal
form of $\tilde{\bo{A}}$, all the matrices $\bar A^{ii}$ are also
Schur, for all $i=1, \cdots, M$.

The dynamics for the comprising subsystems are:
\begin{align}
x^1_{t+1}&= \bar A^{11} x^1_t \label{first}\\
x^i_{t+1}&=\bar A^{ii}x^i_t+ \bar A^{i,i-1} x^{i-1}_t, \; \forall
i=2,\dots,M. \label{second}
\end{align}

\subsubsection{Construction of $X_f^1$}

By taking into account that the
first subsystem is stable and its dynamics are not perturbed by the
other subsystems, the computation of $X_{\text{f}}^1\subset X_1$ as a
positive invariant set with respect to \eqref{first} can be done easily through standard
methods for LTI nominal dynamics available in \cite{RakKerrigan,c18}.

We note also that the boundedness of the set $X_1$ will ensure
boundedness  properties for the set $X_f^1$.

\begin{remark}\label{Lti_pi}
If $X_f^1\subseteq X_1$ is invariant  with respect to \eqref{first}
and $0\in \text{int}(X_f^1)$ then $\alpha X_{\text{f}}^1$ is
invariant and $0\in \text{int}(\alpha_1 X_f^1)$ for any scalar
$\alpha>0$. More than that, if $0<\alpha<1$, then $\alpha
X_f^1\subseteq X_1$.
\end{remark}

\subsubsection{Completing the construction of $X_f$}

For the subsystems $i=2, \dots, M$
 we require a different treatment. If we denote
$\bar A^{i,i-1} x^{i-1}_t=w^i_t$, the dynamics for the remaining
subsystems can be considered as:
\begin{equation}\label{rob2}
x^i_{t+1}= \bar A^{ii}x^i_t + w^i_t, \; \forall i=2, \cdots, M,
\end{equation}
where $w^i_t$ can now be viewed as an unknown disturbance for this
particular subsystem, where $w^i_t$ is bounded, i.e $w_t^i$ is in a
set $\mathcal{W}$\footnote{In the case of the second subsystem
$i=2$, we have $w^2_t= \bar A^{21} x^1_t$, and by taking into
account that $X_{\text{f}}^1$ is positive invariant, it can observed
that $w^2_t$ is bounded, i.e $w^2_t \in \mathcal{W}^2$, where
$\mathcal{W}^2= \bar A^{21} X_{\text{f}}^1$.}. We denote by
$w(\cdot) \in \mathcal{M}_{\mathcal{W}}$ the sequence ${w_0, w_1,
\dots, w_k}$ of disturbances from the admissible set
$\mathcal{M}_{\mathcal{W}}=\left \{ w(\cdot) | w_k \in \mathcal{W},
\; \forall k \in \mathbb{N} \right \}$.

\begin{definition}\label{robust1}
The set ${\cal O} \subseteq X$ is a robust positive invariant set for
a system $x_{t+1}=Ax_t+w_t$, if starting from ${\cal O}$, the evolution of
the system remains in ${\cal O}$ for all $w(\cdot) \in
\mathcal{M}_{\mathcal{W}}$.
\end{definition}
We observe that $X_{\text{f}}^i$ can now be computed as a robust
positive invariant set (RPI) for the subsystem with the index $i\ge
2$, by exploiting the contractiveness properties of $\bar A^{ii}$
and the existence of explicit bounds on $w^i_t$. The practical
construction of such RPI sets is standard in the literature, see for
example the procedures in \cite{RakKerrigan}, \cite{Kof},
\cite{IJC}. In the following such a constructive procedure will be
denoted by $X_{\text{f}}^i = RPI(X^{i},\mathcal{W}^i)$.

\begin{proposition}
Let $X_{\text{f}}^i = RPI(X^i,\mathcal{W}^i)$ be an invariant set with respect to \eqref{second}, having the origin as interior point. There always exists a scalar $0<\alpha<1$ such that $\alpha X_{\text{f}}^i = RPI(\alpha X^{i},\alpha\mathcal{W}^i)$ preserve the invariance properties and additionally  $\alpha X_{\text{f}}^i \subseteq X_i$.
\end{proposition}

\begin{proof} The proof is an immediate application of the Remark \ref{Lti_pi} and the scaling properties of the RPI sets detailed in \cite{Kol}.
\end{proof}

With these (robust) positive invariance and constraint  satisfaction
properties we are able to propose a constructive procedure for
$X_{\text{f}}^i$ in a iterative manner, starting from the first
subsystem and leading to an invariant set in $\rset^n$, as presented
in the following algorithm:
\begin{center}
\begin{enumerate}
\item{compute $X_{\text{f}}^1$}
\item{ \begin{tabbing}
 for \= $i=2:M$   \=  \=  \\
\> 1. compute $\mathcal{W}^{i}= \bar A^{i,i-1}X_{\text{f}}^{i-1}$ ;\\
%\> \> \> b) $\mathcal{W}^i=\mathcal{W}^i \oplus \mathcal{W}^{ij}$ ;\\
\> 2. compute $X_{\text{f}}^i= RPI(X^i,\mathcal{W}^i)$ ;\> \>
 \end{tabbing}
 }
\item{find $0<\alpha <1$ such that $\alpha X_{\text{f}}^i\subset X_i, \forall i=1,\dots, M$}
\end{enumerate}
\end{center}

%\begin{remark} Note that if ${\cal O}$ is a robust positive
%invariant set for the system $x^+=Ax+w$, where $w \in W$, then for
%any $\alpha>0$, the set $\alpha {\cal O}$ is robust positive
%invariant set for the same subsystem, where $w \in \alpha W$. In our
%case we can take $W^{ij} = A^{ij} (\alpha X_{\text{f}}^i)$, where
%$\alpha \leq 1$ in order to guarantee nonemptiness of all the sets
%$X_{\text{f}}^i$.
%\end{remark}
Since for the leader-follower systems described in this paper the
matrix $\bo{A+BK_d}$ is block lower-bidiagonal as well, we can use the
procedure described above to   compute a terminal set of the form $
X_{\text{f}} = \underset{i=1}{\overset{M}{{\large \Pi}}}
X_{\text{f}}^i$ that satisfies the properties A.1--A.3.  Note that
the distributed MPC controller presented below results in a larger
region of attraction  compared to other MPC schemes based on an end
point  constraint \cite{c1}. An additional novelty of our approach
consists in the fact that all the computations for the terminal set
and cost can be carried out in a completely distributed way.
Note that this strategy for constructing sets $X_{\text{f}}^i$ can also be extended to the case
where $\tilde{\bo{A}}$ is block lower triangular, i.e subsystem $i\geq 2$ is affected by
subsystems $1, \dots, i-1$. In this case, the sets $\mathcal{W}^i$ would be constructed
as $\mathcal{W}^i=\mathcal{W}^{i,1}\oplus \dots \oplus \mathcal{W}^{i,i-1}$,
where by $\oplus$ we denote the Minkowski sum: $A \oplus B=\{x+y |
x\in A, \; y \in B\}$ and $\mathcal{W}^{i,j}= \bar A^{ij}X_{\text{f}}^j$, $j=1,\dots,i-1$. \\
We can now reformulate the centralized  MPC problem  for the entire
system \eqref{scheme_first}  as following:
\begin{align} \label{prob}
 V_N(\bo{x})  = &\min_{ \bo{x} , \bo{u}} \sum_{i=1}^M  \sum _{t=0}^{N-1} \ell^i(x_t^i,u_t^i)
  + \ell_\text{f}^i(x_N^i)   \\
\text{s.t:} \; &\text{dynamics \eqref{prob_first} and \eqref{prob_din}}, \nonumber \\
& G_x^i x_t^i  + G_u^i u_t^i \leq b^i, \; G^i x_N^i \leq
 f^i, \; \forall i=1,\dots,M \nonumber
\end{align}
where we assume that the terminal  sets $X_{\text{f}}^i$ constructed
previously   are polyhedra described by $G^i x_N^i \leq f^i$, with
$f^i >0$.

%%%%%%%%%%%%%%%%%%%%%%%%%%%%%%%%%%%%%%%%%%%%%%%%%%%%%%%%%
%%%%%%%%%%%%%%%%%%%%%%%%%%%%%%%%%%%%%%%%%%%%%%%%%%%%%%%%5

\section{Problem restructuring}
\label{restructurare}

We now propose to reformulate problem \eqref{prob} as to obtain a
more suitable structure. We define the intermediary stage variables
for subsystem $i$ as:
\begin{equation}
\mathbf{x}_t^i=\left [(x_t^i)^T \;\ (u_t^i)^T \right ] ^T \in \rset
^{\bo{n}_i}, \nonumber
\end{equation}
where $\bo{n}_i=n_i+m_i$ and $t=1,\dots, N-1$. Next, we define the
general decision variable $\bo{z} \in \rset^{\bo{n}}$ for
\eqref{prob} as follows:
\begin{equation}
\bo{z}=[(\bo{z}^1)^T \dots (\bo{z}^M)^T]^T, \nonumber
\end{equation}
where $\bo{n}=\displaystyle \sum_{i=1}^M N \bo{n}_i$ and
\begin{equation}
\bo{z}^i=[ {(u_0^i)^T \;  (x_1^i)^T  \dots (u_{N-1}^i)^T \;
(x_{N-1}^i)^T \;  (x_N^i)^T} ]^T. \nonumber
\end{equation}

Now, in accordance with the general decision variable as defined
above and in order to  create a more compact and ordered final
structure, we need to define the following matrices:
\begin{eqnarray}
&&\bo{E}^i=\begin{bmatrix} I_{n_i} & 0 \end{bmatrix} \in \rset ^{n_i \times \bo{n}_i}, \;
\bo{A}^i = \begin{bmatrix} A^i & B^i \end{bmatrix} \in \rset ^{n_i \times \bo{n}_i} \nonumber\\
&& \bo{A}^{i,i-1} = \begin{bmatrix} A^{i,i-1} & B^{i,i-1} \end{bmatrix} \in \rset ^{n_i \times \bo{n}_{i-1}} \nonumber \\
&& \bo{Q}^i= \text{diag}(Q^i, R^i) \in \rset ^{\bo{n}_i \times \bo{n}_i} \nonumber\\
&& \tilde{\bo{Q}}^i=\text{diag}({R^i,\bo{Q}^i,\dots,\bo{Q^i},P^i}), \nonumber
\end{eqnarray}
where $\bo{\tilde{Q}}^i \in \rset^{N\bo{n}_i  \times N \bo{n}_i }$ has $N-1$ $\bo{Q}^i$
blocks in its diagonal.
Using the intermediary stage variable  we can rewrite the equality constraints in \eqref{prob}
for subsystem $i$ as:
\begin{equation}\label{eg2_mod}
\bo{E}^i \bo{x}_{t+1}^i= \bo{A}^i \bo{x}_t^i + \bo{A}^{i,i-1} \bo{x}_t^{i-1}.
\end{equation}
We now recast \eqref{prob} as:
\begin{align}{}
&\min_{\bo{z}} \;  \bo{z}^T H \bo{z} \label{Prob} \\
&\text{s.t}:\;  G\bo{z} \leq  b,\; C\bo{z}=c, \nonumber
\end{align}
where $H \in \rset^{\bo{n} \times \bo{n}}$ is
$\text{diag}(\bo{\tilde{Q}}^i)$, with $i=1,\dots,M$. We have
included the equality constraints for each subsystem in \eqref{prob}
in $C\bo{z}=c$,  where $c \in \rset^{\sum_{i=1}^M N n_i}$ and $C \in
\rset ^{\sum_{i=1}^M Nn_i \times \bo{n}}$  are the following:
\begin{equation}
c= \begin{bmatrix} A^1 x_0^1 \\ {\large 0} _{(N-1)n_1,1} \\
A^{21}
x_0^1 + A^2 x_0^2 \\  {\large 0} _{(N-1)n_2,1} \\ \vdots \\
A^{M,M-1} x_{0}^{M-1} + A^M x_0^M \\ {\large 0} _{(N-1)n_M,1}
\end{bmatrix} \nonumber
\end{equation}
\begin{equation} \label{cdef}
C=\begin{bmatrix}C^{11} & 0&\dots&\dots&0 \\ C^{21} & C^{22} &0&
\dots&0 \\ 0 & C^{32} & C^{33} &  \dots &0\\ 0&0&\ddots& \ddots&0 \\
0 &\dots&0& C^{M,M-1} & C^{MM}
   \end{bmatrix}.
\end{equation}
In \eqref{cdef} the matrices $C^{ii} \in \rset^{N n_i \times N
\bo{n_i}}$, for $i=1,\dots,M$,  have the following structure:
\begin{equation*}
C^{ii}=\begin{bmatrix} -B^i & \bo{E}^i &\dots& \dots& \dots & 0 \\
0& -\bo{A}^i  & \bo{E}^i &   \dots& \dots &  0\\ 0 & 0 &  -\bo{A}^i
& \bo{E}^i & \dots & 0 \\  \vdots & \vdots &  \vdots & \ddots &
\ddots & 0 \\ 0 & 0 & 0 & 0 & -\bo{A}^i & I_{n_i}
\end{bmatrix},
\end{equation*}
whilst the matrices $ C^{i,i-1} \in \rset^{Nn_i \times
N\bo{n}_{i-1}}$, for $i=2,\dots,M$, have  the following structure:
\begin{equation*}
\begin{bmatrix} -B^{i,i-1} & 0&\dots& \dots & \dots & 0 \\ 0 & -\bo{A}^{i,i-1}  & 0 & \dots &
\dots & 0 \\ \vdots & \vdots &  -\bo{A}^{i,i-1} & 0 &  \dots & 0\\ \vdots & \vdots & \vdots & \ddots &
\dots & 0  \\ 0 & 0 & 0 & 0 &  -\bo{A}^{i,i-1} & 0
\end{bmatrix}.
\end{equation*}
The inequality constraints in \eqref{prob} have been recast as
$G \bo{z}\leq b$, where $b \in \rset^{\bo{q}}$,  with $\bo{q}=
\displaystyle \sum_{i=1}^M (N q_i +q)$ and $ G \in \rset^{ \bo{q}
\times \bo{n}}$ have the following structure:
\begin{equation}
b=[ (\bo{b}^1)^T, \dots, (\bo{b}^M)^T]^T, \nonumber
\end{equation}
where
\begin{equation}
\bo{b}^i=\big[ (b^i-G_x^i x_0^i)^T ,\overbrace{(b^i)^T \dots
(b^i)^T}^{N-1 \text{ times}},(f^i)^T \big ] ^T  , \nonumber
\end{equation}
and
\begin{equation}
G=\begin{bmatrix} \bo{G}^1 & 0 & \dots & 0 \\ 0 &\bo{G}^2 & \dots & 0 \\
\vdots & \vdots &\ddots  & 0\\ 0& 0 & 0 & \bo{G}^M \end{bmatrix},
\end{equation}
whose matrix blocks $\bo{G}^i \in \rset^{N q_i+q  \times N
\bo{n}_i}$ are:
\begin{equation*}
\bo{G}^i=\begin{bmatrix} G_u^i & 0 & \dots & \dots & \dots & \dots & 0\\
0 & G_x^i & G_u^i & 0&
 \dots & \dots &0 \\ \vdots&0 &0 & \ddots &0& \dots & 0  \\ \vdots &\vdots&\vdots& \vdots & G_x^i & G_u^i & 0\\
  0&0&0&0 &0 & 0 & G^i  \end{bmatrix}
\end{equation*}

%%%%%%%%%%%%%%%%%%%%%%%%%%%%%%%%%%%%%%%%%%%%%%%%

\section{Primal-dual interior point method}
\label{ipm}

Primal-dual interior point methods are very efficient optimization
methods which employ the Karush-Kuhn-Tucker (KKT)  conditions, that
are both necessary and sufficient for achieving optimality for a
convex optimization problem. We intend to use a primal-dual interior
point algorithm for problem \eqref{Prob}, which uses Mehrothra's
predictor-corrector scheme \cite{c26}. The KKT optimality conditions
which result from \eqref{Prob} are:
\begin{eqnarray}
H\bo{z} + C^T \nu + G^T \lambda= 0 \nonumber \\
C\bo{z}-c=0 \nonumber \\
G\bo{z}-b+s=0 \nonumber \\
\Lambda S =0 \nonumber \\
\lambda \geq 0, s \geq 0, \nonumber
\end{eqnarray}
where $s \in \rset^{\bo{q}}$ are slack variables, $\nu \in \rset
^{n_M}$ and $\lambda \in \rset^{\bo{q}}$ are  the Lagrange
multipliers and $S=\text{diag}(s)$, $\Lambda=\text{diag}(\lambda)$
are diagonal matrices formed from the slack variables and respective
multipliers. These conditions lead to the following Newton system
(see \cite{c6} for more details):
\begin{equation}
\begin{bmatrix} H & C^T & G^T & 0\\ C &0&0&0 \\ G &0&0& I \\ 0&0&S&\Lambda \end{bmatrix}
\begin{bmatrix}
\Delta z \\ \Delta \nu \\ \Delta \lambda \\ \Delta s
\end{bmatrix}
= -
\begin{bmatrix}
r_z \\ r_{\nu} \\ r_{\lambda} \\ r_s
\end{bmatrix}.
\end{equation}
We can eliminate $\Delta s$ by using $\Delta s= -\Lambda ^{-1} (r_s+
S\Delta \lambda) $. Furthermore, by reducing $\Delta \lambda =
S^{-1} \Lambda \left ( r_{\lambda} + G\Delta z \right )-S^{-1} r_s$,
we obtain the following system:
\begin{equation}\label{augsys2}
\begin{bmatrix} \Phi & C^T \\ C & 0  \end{bmatrix}
 \begin{bmatrix} \Delta z \\ \Delta \nu \end{bmatrix}
= -\begin{bmatrix} r_d \\ r_{\nu} \end{bmatrix},
\end{equation}
where
\begin{eqnarray} \label{augsys}
\Phi & = & H + G^T S^{-1} \Lambda G \nonumber \\
r_d &=& r_z +G^T S^{-1} \Lambda r_{\lambda}-G^TS^{-1}r_s.
\end{eqnarray}
Next, we form the Schur complement of the matrix in \eqref{augsys2}
so as to obtain the final  system of equations:
\begin{eqnarray}\label{ecprinc}
Y \Delta \nu &=& \tau  \IEEEyessubnumber \label{ecprinc1}\\
Y &=& C \Phi^{-1} C^T  \IEEEyessubnumber \label{ecprinc2}\\
\tau&=& -r_{\nu} -C \Phi^{-1} r_d   \IEEEyessubnumber \label{ecprinc3}\\
\Delta z &=& \Phi^{-1} ( -r_d - C^T \Delta \nu).  \IEEEyessubnumber \label{ecprinc4}
\end{eqnarray}

 Solving \eqref{ecprinc1} would normally employ the
computation of $Y$, which may appear to be overwhelming given the
large dimensions of $Y$ and the fact that it requires an inversion
of $\Phi$. However, due to the way in which matrix $Y$ is formed in
\eqref{ecprinc2}, we show that we can compute its Cholesky
factorization in an efficient and distributed manner, similar to the
one found in \cite{c4} for one linear system. The matrix $\Phi \in
\rset^{\bo{n} \times \bo{n}}$ has a block-diagonal structure $\Phi=\text{diag}\left(\Phi^i\right)$,
%\begin{equation}
%\Phi=\begin{bmatrix} \Phi^1 & 0 & \dots & 0\\ 0 & \Phi^2 & \dots & 0 \\
%\vdots & \vdots &\ddots &0\\ 0 & 0 & 0& \Phi^{M}\end{bmatrix},
%\end{equation}
where the blocks $\Phi^i \in \rset^{N \bo{n}_i \times N \bo{n_i}}$
are also block  diagonal, with their first block of size $m_i \times
m_i$, the following $N-1$ blocks of size $\bo{n}_i \times \bo{n}_i$
and the final block of size $n_i \times n_i$. Now, it can be
observed that resulting matrix $Y$ has the following
block-tridiagonal structure:
\begin{equation*}
\begin{bmatrix} Y^{11} & (Y^{21})^T & 0 & \dots & 0 & 0 \\ Y^{21}  & Y^{22} & (Y^{32})^T &
\dots & 0 & 0 \\ 0 & Y^{32} & Y^{33} & \dots  & 0 & 0\\ \vdots &
\vdots & \vdots & \ddots & \vdots & \vdots \\  0 & 0 & 0 & \dots &
Y_{M-1M-1}& Y_{MM-1}^T \\  0&0&0& \dots & Y_{MM-1} & Y_{MM}
\end{bmatrix}
\end{equation*}
where the matrix blocks are:
\begin{eqnarray}
Y^{11}&=& C^{11} (\Phi^1)^{-1} (C^{11})^T \nonumber \\
Y^{ii}&=& C^{i,i-1} (\Phi^{i-1}) ^{-1} (C^{i,i-1})^T \\ &+& C^{ii} (\Phi^i)^{-1} (C^{ii})^T \text{, }
 \forall i=2 \dots M  \nonumber\\
Y^{i,i-1}&=& C^{i,i-1} (\Phi^{i-1}) ^{-1} (C^{i-1,i-1})^T \text{, } \forall i=2\dots M. \nonumber
\end{eqnarray}

First, we show how to compute efficiently in a distributed fashion
the matrix $Y$. Note that inverting the block components of $\Phi$
and then forming the block components of $Y$ would be very
inefficient. However, if we form the Cholesky factorization of
$\Phi^{i}=\bo{L}^i (\bo{L}^i)^T$ we get:
\begin{eqnarray}
V^{ii}&=& C^{ii} (\bo{L}^i)^{-T} \label{lvar1}\\
W^{i,i-1}&=& C^{i,i-1} (\bo{L}^{i-1})^{-T},\label{lvar2}
\end{eqnarray}
where $\bo{L}^i \in \rset^{N\bo{n}_i \times N\bo{n}_i} $ are also
block diagonal, so that the block components of $Y$ are:
\begin{eqnarray*}
 Y^{11}&=&V^{11}(V^{11})^T\\
Y^{i,i}&=&W^{i,i-1}(W^{i.i-1})^T+V^{ii}(V^{ii})^T, \; \forall i \geq 2\\
Y^{i,i-1}&=&W^{i,i-1} (V^{i-i.i-1})^T, \; \forall i \geq 2
\end{eqnarray*}

\begin{comment}
Computation of $L^{ii}$ and $L^{i,i-i}$ from
\eqref{lvar1} and \eqref{lvar2} by directly inverting  the matrices
$\bo{L}^i$ and then multiplying them with the corresponding matrices
from $C$ would also prove very costly. However, if we take into
account the block diagonal structure of $\bo{L}^i$ and banded
diagonal structure of the block matrices in $C$, the resulting
$V^{ii}$ and $W^{i,i-1}$ matrices are also banded diagonal.
\end{comment}

The most efficient computation of $V^{i,i-1}$ can be done by solving
the following systems of matrix equations, where the matrices
$L_j^i$, with $j=0,\dots, N$, are the diagonal elements of
$\bo{L}^i$:
\begin{align}
L_0^i (V_{11}^i)^T&=(B^i)^T \label{lec1} \IEEEyessubnumber\\
L_j^i (V_{jj}^i)^T&=(\bo{A}^i)^T \text{, } \forall j=1 \dots N-1 \label{lec2} \IEEEyessubnumber\\
L_j^i (V_{j,j+1}^i)^T&=(\bo{E}^i)^T, \forall j=1 \dots N-1 \label{lec3} \IEEEyessubnumber\\
L_N^i (V_{N,N+1}^i)^T&=I_{n_i} \label{lec4}. \IEEEyessubnumber
\end{align}

Equations \eqref{lec1}--\eqref{lec4}  can be efficiently solved by
matrix forward substitution, due to the lower triangular form of
$L_j^i$.  The resulting matrix will take the following form:
\begin{equation*}
V^{ii}= \begin{bmatrix} V_{11}^i & V_{12}^i & 0 & \dots & \dots & 0 \\
0 & V_{22}^i  & V_{23}^i & 0 & \dots &0\\ \vdots & 0 & V_{33}^i & V_{34}^i &\dots & 0\\
\vdots & \vdots &\vdots & \ddots & \ddots & 0\\ 0&0&0&0& V_{N,N}^i & V_{N,N+1}^i \end{bmatrix}.
\end{equation*}
To obtain $W^{i,i-1}$ we solve the following series of matrix
equations,  by matrix forward substitution, considering the dense
nature of $B^{i,i-1}$ and $\bo{A}^{i,i-1}$:
\begin{align}
L_0^{i-1} (W_{11}^i)^T&=(B^{i,i-1})^T \label{lec5}\\
 L_j^{i-1}  (W_{j+1,j+1}^i)^T&=(\bo{A}^{i,i-1})^T \text{, } \forall j=1 \dots N-1, \label{lec6}
\end{align}
where $L_j^{i-1}$ are the diagonal elements of $\bo{L}^{i-1}$. The
resulting  $W^{i,i-1}$ matrices will have a block-diagonal structure.

%\begin{equation}
%W^{i,i-1}=\begin{bmatrix} W_{11}^i & 0 & \dots & \dots & 0 \\ 0 & W_{22}^i &0 & \dots & 0 \\
% \vdots & \vdots & \ddots & 0 & 0\\  0& 0 & 0 & W_{NN}^i  &0 \end{bmatrix}. \nonumber
%\end{equation}

Second, the resulting structure of the Cholesky factorization of
$Y=LL^T$ is as follows:
\begin{equation}
L=\begin{bmatrix} L^{11} & 0 & 0  & \dots & 0 & 0 \\ L^{21} & L^{22} & 0  &
 \dots & 0 & 0 \\ 0 & L^{32} & L^{33} & \dots & 0 & 0 \\
 \vdots & \vdots & \vdots &
\ddots & \vdots & \vdots  \\ 0 & 0 & 0 &  \dots & L^{M-1,M-1} & 0 \\
0 & 0 & 0 & \dots & L^{M,M-1} & L^{MM} \end{bmatrix} \nonumber
\end{equation}
where $L^{ii} \in \rset ^{N n_i \times N n_i} \text{ , } \forall
i=1,\dots,M$ and $L^{i,i-1} \in \rset ^{N n_i \times Nn_{i-1}}
\text{ , } \forall i=2\dots M$. The block components $L^{ii}$ and
$L^{i,i-1}$ can be obtained from the following:
\begin{align}
& L^{11}(L^{11})^T =Y^{11} \label{choly1} \\
& L^{i,i-1} (L^{i-1,i-1})^T = Y^{i,i-1}  \label{choly2} \\
& L^{ii}(L^{ii})^T = Y^{ii}-L^{i,i-1}(L^{i,i-1})^T, \; \forall i
\geq 2. \label{choly3}
\end{align}

Note that the matrices $Y^{ii}$  have a block tridiagonal structure,
but $L^{i,i-1}$ are usually dense so that there is no special
structure in the terms $Y^{ii}-L^{i,i-1}(L^{i,i-1})^T$. Therefore,
the Cholesky factorization of these matrices is computationally
demanding.

%%%%%%%%%%%%%%%%%%%%%%%%%%%%%%%%%%%%%%%%%%%%%%%%%%%%%%%%

\section{Discussion on implementation}
\label{discutie}

The most important aspect of the algorithm previously presented is
that it can be implemented in a
 distributed manner, between the $M$ subsystems.
The Cholesky factorization of $Y$ clearly dominates the system of
equations \eqref{ecprinc} when it  comes to computing cost. By computing the matrices
$V^{i,i}$ and $W^{i,i-1}$, the inversion of $\Phi$ can be avoided, and they can further be used
in \eqref{ecprinc3} and \eqref{ecprinc4} to calculate the respective residuals.
The factorization of $Y$ is also the most complex
when it comes to the communication between subsystems, requiring the
back and forth transmission of matrices between subsystems. Also, due to the structure
of $Y$, the factorization cannot be done in parallel and is achieved
in a sequential manner. For subsystem $i$, with
$i=2, \cdots, M-1$ the following steps are required for obtaining
$L^{ii}$ and $L^{i,i-1}$:
\begin{enumerate}
\item{Compute $\Phi_i=\bo{L}_i\bo{L}_i^T$}
\item{Send $\bo{L}^i$ to subsystem $i+1$}
\item{Receive $\bo{L}^{i-1}$ from subsystem $i-1$}
\item{Compute $V^{ii}$: solve \eqref{lec1} to \eqref{lec4}}
\item{Send $V^{ii}$ to subsystem $i+1$ }
\item{Compute $W^{i,i-1}$: solve \eqref{lec5} and \eqref{lec6}}
\item{Compute $Y^{ii}$, receive $V^{i-1,i-1}$ from subsystem $i-1$}
\item{Compute $Y^{i,i-1}$, receive $L^{i-1,i-1}$ from subsystem $i-1$}
\item{Compute $L^{i,i-1}$ from \eqref{choly2} }
\item{Compute $L^{ii}$ from \eqref{choly3}, send $L^{i,i}$ to subsystem $i$ }
\end{enumerate}
The number of flops for computing the Cholesky factorization of $Y$
by each subsystem  are provided in Table \ref{tabel1}: \\
\begin{table}[h]
\setlength{\belowcaptionskip}{6pt}
\renewcommand{\arraystretch}{0.8}
\caption{\small{Number of flops for computing local components of
the  Cholesky factorization of $Y$} } \centering
\begin{tabular}{c||c}
\hline
\bfseries Operation & Number of flops (approximate)\\
\hline\hline Factor: $\Phi^i =\bo{L}^i (\bo{L}^i)^T$  & $(N-1)
\frac{\bo{n}_i^3}{3}+\frac{n_i^3+m_i^3}{3}$ \\ [3pt] Solve:
\eqref{lec1} & $n_i m_i^2$\\ [3pt] Solve: \eqref{lec2}, \eqref{lec3}
&  $ 2(N-1) n_i
\bo{n}_i^2$\\
[3pt] Solve \eqref{lec4} & $\frac{n_i^3}{3}$\\
[3pt] Solve \eqref{lec5} & $n_i m_{i-1}^2 $ \\ [3pt] Solve
\eqref{lec6} & $(N-1)n_i \bo{n}_{i-1}^2$\\ [3pt] Compute: $Y^{ii}$ &
$N n_i^{2}(n_{i}+\bo{n}_i+\bo{n}_{i-1}+2)$ \\ [3pt] Compute:
$Y^{i,i-1}$ & $N (n_i n_{i-1} \bo{n}_{i-1})$ \\ [3pt]
Compute: $L^{ii}$ & $\frac{N^3n_i^3}{3}$ \\[3pt]
Compute: $L^{i,i-1}$ & $N^3n_{i-1}^{2} n_{i}$ \\
\hline \hline
\end{tabular}
\label{tabel1}
\end{table} \\
It can be observed that the matrices transmitted back and forth are
very sparse, with a known block structure such that the only data
required to be transmitted are these comprising blocks. Also, these
blocks are transmitted only to neighboring subsystems, such that the
transmission of data is localized.

Note that the cost of computing matrices $L^{ii}$ and $L^{i,i-1}$ is
cubic in $N$ but linear in $M$ overall, given the choice of
$\bo{z}$. Also,  computations can be done sequentially and exchange
of information is only between neighbors. If we would rearrange
$\bo{z}$ by the prediction horizon, instead of by subsystems, then
the dominating cost for computing these matrices would be linear in
$N$ overall and cubic in $M$ locally, i.e of order $  \frac{ \left (
\sum_{i=1}^M n_i \right )^3}{3} $ for $L^{ii}$. However, this would
imply that every subsystem has knowledge of the dynamics of all
other subsystems, and as a result computations would require
all-to-all transmission of data between subsystems. Thus, the
efficient choice of $\bo{z}$ given a physical leader-follower system
involves the imposed prediction horizon $N$, the number of the
subsystems $M$ and possible transmission limitations between
subsystems.

%%%%%%%%%%%%%%%%%%%%%%%%%%%%%%%%%%%%%%%%%%%%%%%%%%%%%%%%%%%%%%%%%%%%

\section{Conclusions}
In this paper we have showed that by restructuring certain MPC
problems for large-scale systems we can reduce the computational
cost of implementing an interior point algorithm foe solving such
problems. An analysis for obtaining a stabilizing linear control law
from a distributed viewpoint has been made. By combining several
recent results, we have proved that the online computation of MPC
control laws for some special classes of large scale  systems can be
carried out with increased speed through a reduction of the number
of required flops. This, in combination with ever-increasing
distributed computing power that can be used for distributed
computation of an MPC  law suggests  us that MPC can be used now in
many large-scale applications where it has not been considered
applicable  before.

Further details regarding the efficient transmission of data between
subsystems and the implementation results for the interior point
method presented are omitted for lack of space.

{\small \textbf{Acknowledgements:} \; The research leading to these
results has received funding from: the European Union, Seventh
Framework Programme (FP7/2007--2013) EMBOCON under grant agreement
no 248940; CNCS-UEFISCDI (project TE, no. 19/11.08.2010); ANCS
(project PN II, no. 80EU/2010); Sectoral Operational Programme Human
Resources Development 2007-2013 of the Romanian Ministry of Labor,
Family and Social Protection through the Financial Agreements
POSDRU/89/1.5/S/62557. }

\thispagestyle{empty} \pagestyle{empty}
\end{document}